\documentclass[12pt]{article}

\newcommand{\comment}[1]{}

\usepackage{amssymb,amsthm,amsmath,amsfonts}

\usepackage[linesnumbered,lined,boxed,commentsnumbered]{algorithm2e}
\usepackage{tikz}
\DeclareMathOperator{\PG}{\rm PG}
\DeclareMathOperator{\AG}{\rm AG}

\newtheorem{theorem}{Theorem}[section]
\newtheorem{corollary}[theorem]{Corollary}
\newtheorem{lemma}[theorem]{Lemma}

\theoremstyle{definition}
\newtheorem{definition}[theorem]{Definition}
\newtheorem{example}[theorem]{Example}

\numberwithin{table}{section}

\newcommand{\round}{{\rho}}

\comment{

}
\comment{

}

\title{Burning Steiner Triple Systems}
\author{Andrea C.~Burgess\thanks{Department of Mathematics and Statistics, University of New Brunswick, Saint John, NB, Canada.  andrea.burgess@unb.ca} \and Peter H.~Danziger\thanks{Department of Mathematics, Toronto Metropolitan University, Toronto, ON, Canada. danziger@torontomu.ca} \and Caleb W.~Jones\thanks{Department of Mathematics, Toronto Metropolitan University, Toronto, ON, Canada. caleb.w.jones@torontomu.ca}
\and
Trent G.~Marbach\thanks{Department of Mathematics and Statistics, Acadia University, Wolfville, NS, Canada. trent.marbach@gmail.com} \and David A.~Pike\thanks{Department of Mathematics and Statistics, Memorial University of Newfoundland, St.~John's, NL, Canada.  dapike@mun.ca}}

\begin{document}

\date{\today}
\maketitle

\begin{abstract}
Graph burning is a round-based process which can be viewed as a discrete one-player game that models the spread of influence throughout a network. 
Extending this process to hypergraphs can be done in numerous ways; two such processes that have been studied include the burning process and the lazy burning process for hypergraphs. 
Combinatorial designs can be thought of as hypergraphs with useful  and interesting characteristics. 
In this paper we explore the burning and lazy burning processes in Steiner triple systems (STSs). 
We obtain logarithmic bounds on the burning number of an arbitrary STS and prove the existence of an STS with burning number $\rho$ for any integer $\rho \geq 3$. We observe that the concept of lazy burning is equivalent to the notion of the dimension of an STS, and we consider ramifications of this equivalence. 
We also show that the difference between the burning and lazy burning number of a Steiner triple system can be arbitrarily large.
Finally, we consider burning and lazy burning numbers of affine and projective triple systems in detail.
\end{abstract}



\section{Introduction}
\label{intro}

Burning is a process that models the spread of self-propagating entities such as information or contagion in networks; as is common, we use the analogy of fire to represent the entity.
Typically, burning takes place over a discrete sequence of rounds. In each round, an {\em arsonist} selects and {\em burns} a node or vertex of the network, which then becomes permanently {\em burned}. Additionally, in each round, the fire propagates from burned vertices to adjacent unburned vertices.  All burned vertices remain burned until the end of the process.
The usual goal of the arsonist is to burn every node of the network in as few rounds as possible. 

As noted by Alon in 1992~\cite{Alon1992},
burning arose when Brandenburg and Scott described a problem at Intel to model the transfer of information through processors; their situation was equivalent to burning an $n$-dimensional hypercube graph.  
Independently, in a 2014 paper, Bonato, Janssen, and Roshanbin~\cite{BJR_WAW,BJR_InternetMath} introduced graph burning in general as a combinatorial process which can be applied to any graph; they also coined the term ``graph burning'' and developed much of the early theory on the topic. 
They posed the Burning Number Conjecture, which asserts that any connected graph $G$ on $n$ vertices can be burned in at most $\bigl\lceil \sqrt{n}\: \bigr\rceil$ rounds.
Currently it is known that $b(G) \leq \sqrt{4n/3}+1$, 
where $b(G)$ denotes the fewest number of rounds needed to burn the connected graph $G$,
and also that the Burning Number Conjecture holds asymptotically (i.e., $b(G) \leq (1+o(1)) \sqrt{n}$)~\cite{BBBCKPR_2023,NT_2024}.
For additional background on graph burning, we recommend the survey paper~\cite{Bonato_survey} as well as Chapter~4 of~\cite{Bonato_book}.

The concept of burning has recently been extended from graphs to hypergraphs~\cite{introducing_hyp_burn,mythesis}.
A \emph{hypergraph} $H$ is an ordered pair $(V, E)$ where $V$ and $E$ are disjoint sets, 
and each element of $E$ is a non-empty subset of $V$. The elements of $V=V(H)$ are called \emph{vertices}, and the elements of $E=E(H)$ are called \emph{edges} or \emph{hyperedges}. Hypergraphs are generalizations of graphs, whereby edges can contain any number of vertices; notice that every graph is also a hypergraph in which all edges have cardinality~2.
A hypergraph is $k$\emph{-uniform} if every edge contains exactly $k$ vertices. 
A hypergraph is \emph{linear} if any two distinct edges intersect in at most one vertex.
For more information about hypergraphs, we recommend~\cite{bahm_sanj,hypergraph_text,conn_in_hyp}. 

In the present paper we focus our attention on burning Steiner triple systems.  Steiner triple systems are a well known family of combinatorial designs, which can also be viewed as a family of hypergraphs.
There is a vast literature on this topic; we refer the reader to~\cite{triple_systems_text} for more information on Steiner triple systems specifically and to~\cite{designs_handbook,design_text} for more information on combinatorial designs in general.
For our purposes, a {\em design} is a linear hypergraph  $(V,\cal{B})$ 
in which every pair of vertices from $V$ 
appears in {\em exactly} one edge of $\cal{B}$. 
With this definition, a {\em partial design} is a linear hypergraph in which each pair of vertices appears in at most one edge of $\cal{B}$.

A {\em Steiner triple system} (STS) is a 3-uniform design, and a \emph{partial Steiner triple system} is a 3-uniform linear hypergraph. As is common when talking about designs, and STSs in particular, we will often refer to vertices as {\em points} and edges as {\em blocks} or {\em triples}. 
We denote a Steiner triple system by STS$(v)$, where $v=|V|$ is the number of points and is called the {\em order} of the system.
It is well known that an STS$(v)$ exists if and only if $v\equiv 1,3 \pmod 6$ \cite{kirkman_1847}. Accordingly, if $v\equiv 1,3 \pmod 6$ we will call $v$ \emph{admissible}, and otherwise we will call $v$ \emph{inadmissible}. 
If $G=(X,\mathcal{A})$ and $H=(Y,\mathcal{B})$ are Steiner triple systems 
such that $X \subseteq Y$ and $\mathcal{A} \subseteq \mathcal{B}$, then $G$ is said to be a \emph{subsystem} of $H$ and we write $G \subseteq H$. If $X \subset Y$ and $\mathcal{A} \subset \mathcal{B}$ then $G$ is a \emph{proper subsystem} of $H$ and we write $G \subset H$.

We now give a detailed description of burning in the context of a general hypergraph $H$, which was defined in~\cite{introducing_hyp_burn,mythesis}.
Burning is a round-based process in which the rounds are indexed by positive integers. The set of vertices that are burned by the end of round $\round$ is denoted by $F_H(\round)$; by convention we define $F_H(0)=\emptyset$. For convenience, when $H$ is clear from context we will simply write $F(\round)$.
During each round $\round \geq 1$, the following two things happen simultaneously:

\begin{itemize}
\item 
For each $u \in V(H) \setminus F({\round-1})$, if there exists a non-singleton edge $e \in E(G)$ such that $u \in e$ and $(e \setminus \{u\}) \subseteq F(\round-1)$, then the vertex $u$ becomes burned. Vertices burned in this way are said to be burned by {\em propagation}.
\item The arsonist chooses a vertex $u_\round \in V(H)\setminus F(\round-1)$ and burns it.  The vertex $u_\round$ is called a {\em source}.
\end{itemize}

As previously mentioned, once a vertex is burned, either through the action of the arsonist or by propagation, it remains permanently burned. In the initial round, no vertices become burned through propagation, and the arsonist chooses the first source. 
In subsequent rounds, the arsonist is allowed to choose as a source a vertex that would also become burned due to propagation that round; such sources are called {\em redundant}. Note that it is never advantageous for the arsonist to choose a redundant source if there is another option, but the arsonist may be forced to choose a redundant source in the final round.

During round $\round$, a vertex $u$ is called {\em valid} if $u \notin F(\round-1)$, i.e.\ $u$ is not burned at the end of round $\round-1$, and {\em invalid} otherwise.
Hence the arsonist may only select valid vertices as sources.
Note that every redundant source is also valid. A sequence $(u_1,u_2,\ldots,u_k)$ of vertices is called {\em valid} if $u_\round$ is valid during round $\round$.  A {\em burning sequence} is a valid sequence of sources $(u_1,u_2,\ldots,u_k)$ that leaves the hypergraph completely burned when the arsonist burns $u_\round$ in round~$\round$, for $\round=1,2,\ldots,k$. 
A hypergraph may have  burning sequences of different lengths, as well as different burning sequences of minimum length.
The \textit{burning number} of a hypergraph $H$, denoted $b(H)$, is the earliest round at which $H$ could possibly be completely burned, or equivalently, the length of a shortest burning sequence for $H$. Thus the burning number serves as a measure of how fast the fire can possibly spread to all vertices of $H$.

Note that if a hypergraph $H$ is 2-uniform, and so $H$ is a graph, the hypergraph burning process agrees with the concept of graph burning given in~\cite{BJR_WAW,BJR_InternetMath}, as does the notion of the burning number.
However, in the case of hypergraphs we may define an alternate type of burning called {\em lazy burning} \cite{introducing_hyp_burn,mythesis}. In this case, the arsonist initially chooses a set of points $S\subseteq V$ to be burned.  Subsequently, the fire spreads in each round by propagation but the arsonist does not burn any additional sources. If all of the points eventually become burned, then $S$ is called a {\em lazy burning set}. The smallest size of a lazy burning set for a hypergraph $H$ is called the {\em lazy burning number} of $H$ and is denoted $b_L(H)$. In a graph $G$, $b_L(G)$ is equal to the number of components, but finding $b_L(H)$ for an arbitrary hypergraph $H$ is nontrivial.

Hypergraph burning has received a spate of attention in the recent literature.  The introductory paper~\cite{introducing_hyp_burn} presents various bounds on the burning and lazy burning numbers of a hypergraph and asks a number of open questions for further research.  One of these questions concerns the complexity of computing the lazy burning number, which is shown to be NP-complete in~\cite{BurningMatchingZeroForcing}; further computational aspects are considered in a recent preprint~\cite{AhadiBiniaz}.  The paper~\cite{BurningMatchingZeroForcing} also examines the connection between lazy burning and the zero forcing process; this relationship is further explored in~\cite{AbiadMallee}. 
The paper~\cite{Proportional} introduces alternative rules for propagation based on the proportion of burning vertices in a hyperedge.
The focus of the present paper is burning and lazy burning of a class of combinatorial design, in this case Steiner triple systems.  In this vein, previous work has considered burning hypergraphs formed from Latin squares~\cite{BJMM} and block designs~\cite{AhadiBiniaz}.

When focusing on Steiner triple systems rather than hypergraphs in general, lazy burning has also arisen under some other guises, usually in relation to the subsystem structure of the Steiner triple system. 
For example, the \emph{dimension} of a Steiner triple system $H$, denoted $\dim(H)$, is the maximum integer $d$ such that any set of $d$ points in $H$ is contained within a proper subsystem of $H$. 
See~\cite{Valle_Dukes, HiltonTeirlinck1980}, as well as Chapter~6 of~\cite{ triple_systems_text}, for a comprehensive discussion.
We discuss the relationship between dimension and lazy burning in Section~\ref{sec:LazyBurning}, showing in Theorem~\ref{thm:dim_lazy_equiv} that the lazy burning number of $H$ is exactly $\dim(H)+1$.

Another related topic that has received recent attention is the {\em spreading property} for Steiner triple systems. 
This was introduced by Bl\'{a}zsik and Nagy in~\cite{BlazsikNagy} and was further studied by Nagy and Szemer\'{e}di in~\cite{NagySzemeredi}.  
A \emph{spreading set} of a Steiner triple system $H$ is a subset of the points of $H$ that is not contained by any proper subsystem of $H$.
Nagy and Szemer\'{e}di concentrated their investigation on minimal spreading sets, and established upper bounds on their size.
They also established other results, including showing that if all of the minimal spreading sets of a triple system are large, then the system is necessarily a projective space.
It should be noted that a minimal spreading set for a Steiner triple system $H$ may have larger cardinality than a spreading set of minimum size.  In particular, for each integer $t>3$, there exists a Steiner triple system containing a minimal spreading set of size $3$ (which is necessarily minimum) as well as one of size $t$~\cite{NagySzemeredi}.
Although a lazy burning set for a Steiner triple system is a spreading set, 
the results in the present paper are distinguished from those of~\cite{BlazsikNagy,NagySzemeredi} because our focus is on is on minimum lazy burning sets rather than minimal ones.  

Lazy hypergraph burning is also equivalent to a process known as $H$\emph{-bootstrap percolation}, on a hypergraph $H$, which was introduced in \cite{bootstrap_paper_1}. 
For a uniform hypergraph, the $r$-\emph{line bootstrap percolation process} is a generalization of $H$-bootstrap percolation in which $r$ burned vertices in a hyperedge cause the hyperedge to burn, so lazy burning in a $k$-uniform hypergraph corresponds with the case $r=k-1$. This process has been studied for the $d$-dimensional lattice~\cite{LinePercolation} and for finite projective planes~\cite{percolation_planes}. 
However, results in this area tend to be probabilistic (for example,~\cite{MorrisonNoel2021}) or extremal (for example,~\cite{MorrisonNoel2018}) in nature, and $H$-bootstrap percolation has not been considered in the context of Steiner triple systems. 
Our results, expressed in the language of lazy burning, therefore give results on $2$-line $H$-bootstrap percolation, where $H$ a Steiner triple system.

The remainder of this paper organized as follows.
In Section~\ref{sec:burnSTS} we establish several results about the burning numbers of Steiner triple systems.
In Theorem~\ref{thm:AllBurningNumbers} we prove that for any given integer $\round \geq 3$ there exists a Steiner triple system $H$ with burning number $b(H) = \round$.
Section~\ref{sec:burnSTS} culminates with Theorem~\ref{thm:ConfinedBurningNumber}, which gives upper and lower bounds the burning number of an arbitrary STS$(v)$.

In Section~\ref{sec:LazyBurning} we consider lazy burning of Steiner triple systems. We review the literature on dimension and discuss its connection to lazy burning. In particular, this gives a logarithmic bound on the lazy burning number of a Steiner triple system (Theorem~\ref{lazy_upper_bound_sts}) and shows that for any integer $\ell \geq 3$, an STS$(v)$ with lazy burning number $\ell$ exists whenever $v$ is a sufficiently large admissible order (Theorem~\ref{thm:AJWH1977}).  We further discuss lazy burning of a triple system in the context of its subsystems and show that the difference between the burning and lazy burning number of a Steiner triple system can be arbitrarily large.

In Section~\ref{sec:Affine and Projective Geometries} we consider affine and projective triple systems specifically. The lazy burning numbers of these designs have previously been considered in the context of dimension.  Hilton and Teirlinck~\cite{HiltonTeirlinck1980} noted their values without proof and Zeitler~\cite{Zeitler1987} gave a formal proof in the projective case.  We give a proof for both cases using the language of lazy burning.  Moreover, we determine the burning numbers of affine and projective triple systems.

In Section~\ref{sec:Discussion} we conclude with some discussion and open questions.

\section{Burning Steiner Triple Systems}
\label{sec:burnSTS}

In this section, we examine the behaviour of the round-based burning process on Steiner triple systems.  For the purposes of the following results, denote the total number of vertices that are burned in a hypergraph $H$ at the end of round $\round$ by $f_H(\round)$, so $f_H(\round) = |F_H(\round)|$.   
We use $f'_H(\round)$ to denote the total number of edges whose vertices are all burned at the end of round $\round$, and refer to such edges as being {\em fully burned} at round $\round$ and thereafter.  
For the sake of simplicity, we may often write $f_H(\round)$ and $f^\prime_H(\round)$ as $f(\round)$ and $f^\prime(\round)$ 
respectively when $H$ is clear from context. 

\begin{lemma}
\label{first4rounds}
Let $H$ be an STS$(v)$ with $v > 7$. For any valid  burning sequence on $H$, if the arsonist does not choose a redundant source in the first four rounds, then $f(1)=1$, $f(2)=2$, $f(3)=4$, and $f(4)=8$. Also, $f^\prime(1)=f^\prime(2)=0$, $f^\prime(3)=1$, and $4\leq f^\prime(4)\leq7$.
\end{lemma}

\begin{proof}
Let $H=(V,{\cal B})$ and
observe Figure \ref{ex_3}, in which the first four sources in a burning sequence for an arbitrary STS are $u_1$, $u_2$, $u_3$, and $u_4$. Clearly only the first source $u_1$ is burned in round 1, so $f(1)=1$ and $f^\prime(1)=0$. In round 2, there is no propagation and the arsonist picks a second source, $u_2$. Since $u_1$ and $u_2$ are the only vertices that are burned at the conclusion of round 2, we have $f(2)=2$ and $f^\prime(2)=0$. 

Let $a$ be the unique point of $V$ such that $\{u_1,u_2,a\}$ is a triple in ${\cal B}$, and so in round $3$, $a$ is burned through propagation and the arsonist picks a non-redundant source $u_3 \neq a$. In round 3, $u_1,u_2,u_3$, and $a$ are burned, and the triple $\{u_1,u_2,a\}$ is fully burned, so $f(3)=4$ and $f^\prime(3)=1$. 

Now, the source $u_3$ appears in a distinct triple of ${\cal B}$ for each of $u_1,u_2$, and $a$. Let these edges be $\{u_3,u_1,b\}$, $\{u_3,u_2,c\}$, and $\{u_3,a,d\}$. Note that $b$, $c$, and $d$ are all distinct from each other and from the vertices $u_1$, $u_2$, and $a$, as depicted in Figure \ref{ex_3}. The arsonist now chooses some vertex $u_4$ not in the set $\{ u_1,u_2,u_3,a,b,c, d\}$ as a source in this round, so $f(4)=8$. Observe that it is possible that $\{u_1,u_2,a\}$, $\{u_3,u_1,b\}$, $\{u_3,u_2,c\}$, and $\{u_3,a,d\}$ are the only fully burned triples at the end of round 4, so $4\leq f^\prime(4)$. 
The only other triples that could possibly be fully burned are  
$\{u_1,c,d\}$, $\{u_2,b,d\}$, and $\{a,b,c\}$, if they exist within $\cal B$,
 and hence $f^\prime(4)\leq7$.
\end{proof}

\begin{figure}[ht]

\centering
\begin{tikzpicture}[scale=1.1]

\node (u1) at (0.07,1) {};
\fill [fill=black] (u1) circle (0.1) node [below] {$u_1$};
\node (u2) at (1,0) {};
\fill [fill=black] (u2) circle (0.1) node [right] {$u_2$};   
\node (u3) at (0,-1) {};
\fill [fill=black] (u3) circle (0.1) node [right] {$u_3$};    
\node (u4) at (1.85,1.6) {};
\fill [fill=black] (u4) circle (0.1) node [right] {$u_4$};
    
\node (v1) at (-1,0) {};
\fill [fill=black] (v1) circle (0.1) node [right] {$a$}; 
\node (v2) at (2,-1) {};
\fill [fill=black] (v2) circle (0.1) node [left] {$c$}; 
\node (v3) at (-2,-1) {};
\fill [fill=black] (v3) circle (0.1) node [right] {$d$}; 
\node (v4) at (0,1.75) {};
\fill [fill=black] (v4) circle (0.1) node [right] {$b$}; 

\draw [line width=0.03cm] [rounded corners=0.7cm] (-0.05,1.5)--(2.1,-0.25)--(-1.6,-0.3)--cycle;
\draw [line width=0.03cm] [rounded corners=0.7cm] (1.1,0.6)--(2.55,-1.3)--(-0.6,-1.3)--cycle;
\draw [line width=0.03cm] [rounded corners=0.7cm] (-1.05,0.6)--(0.94,-1.4)--(-2.6,-1.3)--cycle;
\draw [line width=0.03cm] [rounded corners=0.5cm] (-0.4,2.15)--(0.6,2.15)--(0.7,-1.7)--(-0.4,-1.7)--cycle;
\end{tikzpicture}

\caption{The first four edges that burn when burning an STS($v$).}
\label{ex_3}
\end{figure}
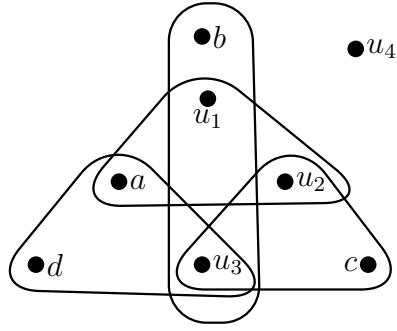

In Theorem \ref{small_lem} we establish a lower bound on the burning number of an arbitrary STS.
Towards this goal, we define the recursive function $h(\round)$ by $h(1)=1$, and for each $\round \geq 2$,
\begin{equation}
\label{defn_of_h}
h(\round)=\binom{h(\round-1)}{2} -2h(\round-1)+3\round-2.
\end{equation}
The first few values of $h(\round)$ are given in Table~\ref{h_g_table}.
We now show that $h(\round)$ is an upper bound on the total number of vertices that could be burned by the end of round $\round$.

\begin{table}[ht]
\centering
\begin{tabular}{|c|c|}
\hline
$\round$ & $h(\round)$   \\ \hline 
1 & 1          \\ \hline  
2 & 2          \\ \hline  
3 & 4          \\ \hline  
4 & 8          \\ \hline  
5 & 25         \\ \hline  
6 & 266        \\ \hline  
7 & 34\,732      \\ \hline 
8 & 603\,069\,104  \\ \hline  
9 & 181\,846\,170\,592\,008\,673 \\ \hline 
\end{tabular}
\caption{The first nine values for the function $h$.}
\label{h_g_table}
\end{table}

\begin{lemma}
\label{big_lem}
Let $(V,\mathcal{B})$ be an STS$(v)$. For any burning sequence, in any round $\round \geq 1$, $f(\round) \leq h(\round)$. 
\end{lemma}

\begin{proof}
Consider round $\round$ of the burning process on an arbitrary STS$(v)$, $(V,\mathcal{B})$.
Observe that since exactly one point is chosen to be burned per round, exactly $\round$ vertices have been chosen from the burning sequence. Thus, by round $\round$, there must be at least $f(\round) - \round$ points burned through propagation, that were not part of the burning sequence.  
To each point $x$ that was burned through propagation, we may associate a triple of the form $\{x,y,z\}$ such that $y$ and $z$ are points that were burned in rounds strictly before $x$ was burned. 
Note that any such triple can be associated with at most one point, namely the point burned last of those in the triple. 
There is exactly one associated triple for each point burned through propagation. Each of these triples is burned, and so the number of burned triples is greater than the number of points burned through propagation, that is
$f^\prime(\round)\geq f(\round) - \round$. 

We now proceed by induction on $\round$.
In round 1, $f(1)=h(1)=1$.
Now suppose we have completed round $\round$, and consider round $\round+1$.  As the inductive hypothesis, we assume $f(\round) \leq h(\round)$. 
In round $\round+1$, there will be $f(\round)$ points that were burned in previous rounds. 
For an unburned points $x$ to be burned through propagation in round $\round+1$, there must be two previously burned points, say $y$ and $z$, with $\{x,y,z\} \in \mathcal{B}$. 
There are $\binom{f(\round)}{2}$ pairs of burned points at the start of round $\round+1$. 
Of these pairs of burned points, exactly $3f^\prime(\round)$ occurred in triples with all three points already burned. 
This means at most $\binom{f(\round)}{2}-3f^\prime(\round)$ points will be burned due to propagation in round $\round+1$. 
An additional point is chosen to be burned as a source, so we have $f(\round+1) \leq f(\round) + \binom{f(\round)}{2}-3f^\prime(\round)+1$. 

We therefore find that 
\begin{align*}
    f(\round+1) &\leq f(\round) + \binom{f(\round)}{2}-3(f(\round) - \round) +1\\
            &= \binom{f(\round)}{2}-2f(\round) +3\round+1\\
            & \leq \binom{h(\round)}{2}-2h(\round) +3\round+1 \\
            & = h(\round+1),
\end{align*}
where the last inequality holds as 
\[
\binom{f(\round)}{2}-2f(\round) = f(\round)\left( \frac{f(\round)-1}{2} -2 \right) \leq h(\round)\left( \frac{h(\round)-1}{2} -2 \right). \qedhere
\]
\end{proof}

We now give a lower bound on the burning number of an arbitrary STS$(v)$.

\begin{theorem}
\label{small_lem}
Let $H$ be a Steiner triple system on $v$ vertices. 
If $\round_v$ is the smallest natural number satisfying $v\leq h(\round_v)$, then $\round_v \leq b(H)$. 
\end{theorem}

\begin{proof}
By Lemma~\ref{big_lem}, in any STS($v$), no more than $h(\round)$ vertices can possibly be on fire at the end of round $\round$. Since $h(\round_v -1)<v$, a burning sequence of length $\round_v -1$ cannot possibly burn all $v$ vertices of the STS. Thus, a sequence of length $\round_v$ or greater is required, so $\round_v\leq b(H)$.
\end{proof}

We now use this result to derive an explicit bound on the burning number of an STS$(v)$ by confining the logarithm of the function $h$ within a narrow interval.  

\begin{theorem}
If $\round \geq 9$, $\alpha = 1.07925087759784$, and 
$\beta = 1.07925087759785 = \alpha+10^{-14}$, then
    $$2 \alpha^{2^{\round}}+\frac{5}{2} \leq h(\round) \leq 2\beta^{2^{\round}}.$$
\label{Size of h(r)}
\end{theorem}
\begin{proof}
We proceed by induction on $\round$.  
For the base case for the upper bound, consider $\round=9$.  It is straightforward to verify that $\lfloor 2 \beta^{2^\round}\rfloor =181846170592709201$, which exceeds $h(9)$; see Table~\ref{h_g_table}. 

Now, recalling that $h(\round)=\binom{h(\round-1)}{2}-2h(\round-1)+3\round-2$, we have the inductive step, as 
\begin{align*}
h(\round) &= \binom{h(\round-1)}{2} -2h(\round-1)+3\round-2 \\
     &= h(\round-1)\left( \frac{h(\round-1)-1}{2} - 2 \right) +3\round-2 \\
     &\leq 2\beta^{2^{\round-1}} \left( \beta^{2^{\round-1}} - \frac{5}{2}\right)+3\round-2\\
     &= 2\beta^{2^{\round-1}\cdot 2} - 5 \beta^{2^{\round-1}} +3\round-2\\
     & \leq 2\beta^{2^{\round-1}\cdot 2}\\
     & = 2\beta^{2^{\round}},
\end{align*}
where the last inequality holds since $5 \beta^{2^{\round-1}} \geq 3\round-2$ for $\round \geq 5$. This completes the inductive step, and the proof for the upper bound. 

For the lower bound, we consider as the base case $\round=8$. It follows from calculation that $\lceil 2 \alpha^{2^\round}+5/2\rceil = 603069104$. 
This is equal to $h(8)$; see Table~\ref{h_g_table}.

For the inductive step, we have 
\begin{align*}
h(\round) &=\binom{h(\round-1)}{2} -2h(\round-1)+3\round-2 \\
     &= h(\round-1)\left( \frac{h(\round-1)-1}{2} - 2 \right) +3\round-2 \\
     &\geq \left( 2 \alpha^{2^{\round-1}} +\frac{5}{2} \right)\left( \alpha^{2^{\round-1}}-\frac{5}{4}\right)+3\round-2\\
     &= 2 \alpha^{2^{\round-1}\cdot 2} - \frac{25}{8} +3\round-2\\
     & \geq 2 \alpha^{2^{\round}}- \frac{5}{2}+3\round-3+\frac{3}{8}\\
     & \geq 2 \alpha^{2^{\round}}-\frac{5}{2},
\end{align*}
where we have $3\round\geq 3-\frac{3}{8}$ for $\round \geq 1$. 
This completes the inductive step, and the proof for the lower bound.
\end{proof}

\begin{corollary}\label{cor:lowerbound_burning_logs}
    If $H$ is an STS$(v)$ and $\beta = 1.07925087759785$, then $b(H) > \lfloor \log_2(\log_\beta(v/2))\rfloor$.      
\end{corollary}

\begin{proof}
Assume by way of contradiction that $b(H) \leq \lfloor \log_2(\log_\beta(v/2))\rfloor$. Then, because $v$ is necessarily odd, $\log_2(\log_\beta(v/2)) \notin \mathbb{N}$. 
We therefore have $b(H) <  \log_2(\log_\beta(v/2))$. Using elementary arithmetic, we equivalently have $2\beta^{2^{b(H)}}<v$. Finally, by applying Theorem~\ref{Size of h(r)} with $\rho=b(H)$, we get the series of inequalities $$h(b(H))\leq 2\beta^{2^{b(H)}}<v, $$ which is a contradiction since, by the definition of $h$, $h(b(H))\geq v$ must be true.
\end{proof}

Corollary~\ref{cor:lowerbound_burning_logs} implies that the burning number of a Steiner triple system becomes arbitrarily large, and hence
for any fixed $b\in\mathbb{N}$, there can only be finitely many STSs with burning number $b$. In fact, as we will see from Theorem~\ref{thm:AllBurningNumbers}, 
each integer $b \geq 3$ is attained as the burning number of some Steiner triple system.  
To prove this result, we will need the following result of Horsley regarding embedding Steiner triple systems.

\begin{theorem}\cite{MR3248026} \label{thm:partialSTS_smaller}
    Any partial Steiner triple system of order $u \geq 62$ with at most $\frac{u^2}{50} - \frac{11u}{100} - \frac{116}{75}$ triples can be embedded into some STS$(v)$ for each integer $v \geq  \frac{8u+17}{5}$ such that $v \equiv 1, 3 \pmod 6$.
\end{theorem}

\begin{theorem}\label{thm:AllBurningNumbers}
For any integer 
$\round \geq 3$, 
there exists an STS$(v)$,
$H$, satisfying $h(\round-1)+1 \leq v \leq 2h(\round-1)-3$ such that  $b(H)=\round$.   
\end{theorem}
\begin{proof}
We have computationally verified that the Steiner triple systems of orders 3, 7 and 9 have burning numbers 3, 4 and 5, respectively, while the projective Steiner triple system on 31 points has burning number 6; this latter result will also be proven theoretically in Corollary~\ref{cor:proj_regular_burning_solved}.

Fix $\round\geq 7$. We will construct a partial Steiner triple system $H_{\round-1}=(V_{\round-1},\mathcal{B}_{\round-1})$.
To do this, we recursively construct partial systems $H_i = (V_i, \mathcal{B}_i)$ for $0\leq i< \round$. Set  $V_0=\mathcal{B}_0 = \emptyset$, and for $i\geq 0$, we define $(V_{i+1}, \mathcal{B}_{i+1})$ by the following process.
For each pair of points $x,y \in V_i$ that do not occur together in a triple of $\mathcal{B}_i$, we create a new point $z_{xy}$ that is added to $V_{i+1}$ and we add $\{x,y,z_{xy}\}$ to $\mathcal{B}_{i+1}$. 
We then add an additional new point $z_{i+1}$ to $V_{i+1}$. 

Note that since $\mathcal{B}_i$ contains $3|\mathcal{B}_i|$ pairs of points from $V_i$, there are $\binom{|V_i|}{2} -3|\mathcal{B}_i|$ of pairs of points \emph{not} in any block of $H_i$. Each of these pairs generates a new point on the next step, so
\begin{equation}
\label{V_recursion}
|V_{i+1}| = |V_i| + \binom{|V_i|}{2} -3|\mathcal{B}_i| + 1
\end{equation}
and
\begin{equation*}
\label{B_recursion}
|\mathcal{B}_{i+1}| = |\mathcal{B}_i| + \binom{|V_i|}{2} - 3|\mathcal{B}_i|.
\end{equation*}
For each $i \geq 0$, we consider the differences $\Delta_{i} =  |V_{i}|-|\mathcal{B}_{i}|$. Observe that
$$\Delta_{i+1} = |V_{i+1}|-|\mathcal{B}_{i+1}| = |V_i|-|\mathcal{B}_i| + 1 = \Delta_i + 1.$$
Since $\Delta_0=0$, it follows that $\Delta_i=i$, thus, $|V_{i}|-|\mathcal{B}_{i}| = i$. Substituting this into Equation~(\ref{V_recursion}), we obtain
\[
|V_{i+1}| = |V_i| +\binom{|V_i|}{2} -3(|V_i|-i) + 1
= \binom{|V_i|}{2} -2|V_i| +3(i+1) -2 = h(i+1).
\]

Observe that since $\round\geq 7$,
$|V_{\round-1}| = h(\round-1) \geq h(7-1)=266 \geq 62$.
Now, since $u \leq \frac{1}{50}u^2 - \frac{11}{100}u - \frac{116}{75}$ holds whenever $u\geq 57$, we have that
\[
|\mathcal{B}_{\round-1}| = |V_{\round-1}|-(\round-1) \leq \frac{1}{50}|V_{\round-1}|^2 - \frac{11}{100}|V_{\round-1}| - \frac{116}{75}. 
\]
Hence, by Theorem~\ref{thm:partialSTS_smaller}, $H_{\round-1}$ can be embedded in a Steiner triple system, $H=(V,\mathcal{B})$, whose order $v$ is at most $\frac{8|V_{\round-1}|+17}{5} + 4 < 2(|V_{\round-1}|-1)$, where the inequality holds since $|V_{\round-1}| \geq 24$. 

Note that the Steiner triple system $H$ has order strictly greater than $h(\round-1)$, and so will take at least $\round$ rounds for all vertices to be burned, due to Lemma~\ref{big_lem}. 
As such, we have $b(H) \geq \round$. 

We now perform the burning process on $H$ with the first $\rho-1$ sources being $(z_1 ,\ldots, z_{\round-1})$. We clearly have that the vertices in $V_{\round-1}$ will be burned by round $\round-1$. 

Recall that $V$ is the point set of the Steiner triple system $H$, and $v=|V| < 2(|V_{\round-1}|-1)$. We partition $V$ into the three point sets; $V_{\round-1}$ (which contains $z_{\rho-1}$), $X\subseteq V\setminus V_{\round-1}$, and $Y\subseteq V\setminus V_{\rho-1}$, such that $X$ and $Y$ partition $V\setminus V_{\rho-1}$, where $x \in X$ if and only if there are $a,b\in V_{\round-1}$ with  $\{a,b,x\}$ a triple of $H$. That is, the elements of $X$ are not burning by the end of round $\round-1$, but will catch fire through propagation during round $\round$. We will show that $Y=\emptyset$, proving that $H$ is fully burned by the end of round $\rho$ (in which a redundant source must be chosen).

For a contradiction, suppose that $Y \neq \emptyset$, and fix $y \in Y$.  For each $a \in V_{\round-1}$, there must be a unique point $w_a \in V$ so that the triple $\{y,a,w_a\}$ is a triple of $H$.  Since $y \notin X$, it must be that $w_a \notin V_{\round-1}$.  Thus, $\{w_a : a \in V_{\round-1} \} \subseteq V \setminus V_{\round-1}$. Since $\{w_a : a \in V_{\round-1} \}$ must have cardinality $|V_{\round-1}|$, this means that $|V \setminus V_{\round-1}| \geq |V_{\round-1}|$, and so $|V| \geq 2|V_{\round-1}|$.  However, we know that $|V|<2|V_{\round-1}|-2$, so we have derived a contradiction.  
\end{proof}

Next we present an upper bound for the burning number of an STS$(v)$. An analogous result for lazy burning is known from the literature, see Theorem~\ref{lazy_upper_bound_sts}.

\begin{theorem}
\label{upper_bound_sts}
If $H=(V,\mathcal{B})$ is an STS(v), then $b(H)\leq \lceil \log_2(v)\rceil+1$.
\end{theorem}
\begin{proof}
Recall that $F(\round)$ is the set of points that are on fire at the end of round $\round$.  We will show by induction that there is a burning sequence such that $|F(\round)|\geq\min\left\{v,2^{\round-1}\right\}$ for each $\round\geq 1$.
    It then naturally follows that all points are burned by round $\lceil \log_2(v)\rceil+1$. 
    
    For the base case, observe that $|F(1)|=1=2^{1-1}$. 
    In order to inductively define a burning sequence we assume that on round $\round$ not all vertices have been burned; otherwise $|F(\round)|=v$. We thus assume as the inductive hypothesis that 
    $|F(\round)| = 2^{\round-1}$. 
    As the source for round $\round$, 
    we choose a point $s_\round \in V$ such that $s_\round$ will not be burned by propagation in the next round, if there is such a point. 
    If there is no such point, then all the unburned points will be burned during the propagation of round $\round+1$, so choosing a redundant source in round $\round+1$ completes the burning sequence. 
    
    Now, assuming that the sequence is not finished, we consider the propagation that occurs in round $\round+1$. 
    For each $x\in F(\round)\setminus\{s_\round\}$, consider the triple $\{s_\round ,x,y_x\}$ and define
    \[
    T_\round = \{y_x : x\in F(\round)\setminus\{s_\round\}\}.
    \]
    Note that the $y_x$ are distinct, since each $y_x$ appears in exactly one triple with $s_\round$, 
    so $|T_\round|=|F(\round)\setminus\{s_\round\}|$.
    If $y_x\in F(\round)$, then $s_\round$ would have burned by propagation one round after it was chosen (i.e.,~in round $\round+1$), which contradicts how the sources were chosen. So, it must be the case that $T_\round \subseteq V \setminus (F(\round) \cup \{s_\round\})$; it follows that 
    the points in $T_\round$  will burn during the propagation step of round $\round+1$ due to the triples $\{s_\round,x,y_x\}$. 
    Counting the source that is chosen in round $\round+1$, the number of points that are on fire at the end of round $\round+1$ is at least 
    \[
    |F(\round)|+|T_\round|+1 \geq (2^{\round-1})+(2^{\round-1}-1)+1=2^{\round},
    \]
     completing the induction. 

    Therefore $|F(\round)| \geq \min\{v,2^{\round-1}\}$.
    The burning sequence described above has $\tau$ rounds, where $\tau$ is the largest integer such that $2^{\tau-1}\leq v$.
    Since $2^{\tau-1}\leq v$ implies that $\tau\leq\lceil \log_2(v)\rceil+1$, we must then have that 
    $b(H)\leq \tau\leq\lceil \log_2(v)\rceil+1$.  \end{proof}

Combining Corollary~\ref{cor:lowerbound_burning_logs} with Theorem~\ref{upper_bound_sts} yields the following result. 

\begin{theorem}
\label{thm:ConfinedBurningNumber}
If $H$ is an STS$(v)$ and $\beta = 1.07925087759785$, then
\[\lfloor \log_2(\log_\beta(v/2))\rfloor +1\leq b(H)\leq   \lceil \log_2(v)\rceil+1.\]
\end{theorem}

We note that the upper bound of Theorem~\ref{thm:ConfinedBurningNumber} is attained projective triple systems (as we will see from Theorem~\ref{cor:proj_regular_burning_solved}), and the lower bound is attained by the triple system constructed in Theorem~\ref{thm:AllBurningNumbers}, up to a small additive constant. 
Consider a triple system $\widehat{H}$ constructed in Theorem~\ref{thm:AllBurningNumbers} with order $v \geq h(\rho-1)+1$.  Using that $h(\rho-1) \geq 2 \alpha^{2^{\rho-1}}$ from Theorem~\ref{Size of h(r)} and rearranging gives $v/2-7/4 \geq \alpha^{2^{\rho-1}}$, so that $v/2 > \alpha^{2^{\rho-1}}$.  Hence $\log_2(\log_{\alpha}(v/2)) > \rho-1$, so that $b(\widehat{H}) = \rho < \log_2(\log_{\alpha} (v/2)) + 1$.  Thus by Theorem~\ref{thm:ConfinedBurningNumber}, 
\[
\lfloor \log_2(\log_{\beta}(v/2))\rfloor+1 \leq b(\widehat{H}) \leq \lfloor\log_2(\log_{\alpha} (v/2))\rfloor + 1,
\]
and the upper and lower bounds here differ by at most $1$.
In particular, the lower bound can also be written as 
\[
\lfloor \log_2(\log_\beta(v/2))\rfloor +1 = \lfloor \log_2(c\log_2(v/2))\rfloor +1 = \lfloor \log_2(\log_2(v/2))+\log_2(c)\rfloor +1,
\]
where $c=\frac{1}{\log_2(\beta)} \approx 9.08840846208484$ and $\log_2(c) \approx 3.18402767561964$.
On the other hand, the upper bound is 
 \[
 \lfloor \log_2(\log_\alpha(v/2))\rfloor+1 = \lfloor \log_2(c_1\log_2(v/2))\rfloor+1 = \lfloor \log_2 (\log_2(v/2)) + \log_2(c_1)\rfloor+1,
 \]
 where $c_1= \frac{1}{\log_2(\alpha)} \approx 9.08840846208595$ and $\log_2(c_1) \approx 3.18402767561982$.
Thus the value of $b(\widehat{H})$ can exceed the lower bound of Theorem~\ref{thm:ConfinedBurningNumber} by at most $1$.

\section{Lazy Burning and Dimension}
\label{sec:LazyBurning}

In this section we turn our attention to the lazy burning process.  Before explicitly focusing on Steiner triple systems, we develop a simple polynomial-time algorithm that, when given a hypergraph $H$, produces a lazy burning set $S$ for $H$.  Algorithm~\ref{lazy_alg} below accomplishes this task by choosing any unburned vertex as a source, setting it on fire, and then allowing the fire to propagate until it cannot spread any further on its own; these steps are repeated until the hypergraph is fully burned.  This approach constructs a lazy burning set $S$ in a manner that avoids adding extraneous vertices to the lazy burning set that is built.  Hence $S$ is a minimal lazy burning set (i.e., there is no lazy burning set $S'$ that is a proper subset of $S$), although not necessarily minimum.  However, the algorithm does produce minimum lazy burning sets for some families of input hypergraph, such as projective triple systems.  

\begin{algorithm}[ht]

\SetKwInOut{Input}{Input}\SetKwInOut{Output}{Output}

\Input{A hypergraph $H=(V,E)$}

\Output{A lazy burning set for $H$}

\BlankLine

Initialize $S$ to be an empty set

\Repeat{$H$ is fully burned}{

Select any valid source vertex, burn it and add it to $S$

\Repeat{no vertex catches fire through propagation}{

Propagate the fire

}

}

\Return{$S$}
\text{(a lazy burning set)}

\caption{Na\"{\i}vely construct a lazy burning set}
\label{lazy_alg}
\end{algorithm}

\subsection{Lazy Burning and Subsystems}

We now turn our attention specifically to the lazy burning process on Steiner triple systems.  
As a first observation, suppose the arsonist burns some subset $S$ of vertices of a Steiner triple system $H$; 
after the fire propagates, the vertices which have burned are precisely those of some subsystem of $H$.  
This can be seen by noting that for any two burning vertices $x$ and $y$, the third vertex of the block containing $x$ and $y$ must burn by propagation.  
With respect to Algorithm~\ref{lazy_alg} applied to a Steiner triple system, this observation means that each time the inner loop has finished, the set of burning vertices induces a subsystem of the input system.

The following result is now immediate, and characterizes lazy burning sets in terms of the existence of proper subsystems.  

\begin{theorem}
    Let $H$ be a Steiner triple system.  A set $S$ of vertices in $H$ is a lazy burning set for $H$ if and only if there is no proper subsystem of $H$ that contains $S$.
\end{theorem}

As we will see in the next theorem, the subsystem whose vertices burn if the arsonist burns a set of vertices $S$ is the {\em smallest} subsystem of $H$ containing $S$.  To this end, we make the following definition.
\begin{definition}
\label{def:generated subsystem}
    Let $H=(V,\mathcal{B})$ be an $\mathrm{STS}(v)$ and let $S \subseteq V$. 
    The {\em subsystem generated by $S$}, written $\langle S \rangle$,  is the unique minimal subsystem of $H$ 
    that contains $S$.  
    Equivalently, $\langle S \rangle$ is the intersection of all subsystems of $H$ that contain $S$.  
    We say that $S$ {\em generates} $\langle S \rangle$. 
\end{definition}

\begin{theorem} \label{SmallestSubsystemBurns}
Suppose the arsonist burns a set $S$ of vertices in a Steiner triple system $H=(Y,\mathcal{B})$, and let $(Z,\mathcal{C}) = \langle S \rangle$ be the subsystem of $H$ generated by $S$. The fire will propagate until precisely the vertices in $\langle S \rangle$ have burned.
\end{theorem}

\begin{proof}
    As previously noted, the fire will propagate to the vertices of a subsystem of $H$, say $(X,\mathcal{A})$.  We suppose that $(X,\mathcal{A})$ is distinct from $(Z,\mathcal{C})$, and seek a contradiction.  Since $(X,\mathcal{A})$ is a subsystem of $H$ containing $S$, it must contain $(Z,\mathcal{C})$ as a proper subsystem.  Thus no vertex of $X \setminus Z$ can be in a block $B \in \mathcal{B}$ with two vertices of $Z$.  Clearly the fire must propagate to the vertices of $Z$; however, it then cannot spread to any vertex of $X \setminus Z$, yielding a contradiction. 
\end{proof}

Theorem~\ref{SmallestSubsystemBurns} allows us to establish a connection between lazy burning and the dimension of a Steiner triple system. 
Dimension has previously been considered in the literature; see, for instance~\cite{Hilton1977, HiltonTeirlinck1980, TeirlinckThesis, Teirlinck1979, Zeitler1985, Zeitler1987}.

\begin{definition} 
\label{def:dimension}
    The {\em dimension} of a Steiner triple system $H$, denoted $\dim(H)$, is one less than the minimum cardinality of a subset that generates $H$.
\end{definition}
Definition~\ref{def:dimension} generalizes the dimension of a projective or affine triple system.  Note that the dimension of a Steiner triple system $H$ is the maximum integer $d$ such that any set of $d$ vertices in $H$ is contained within a proper subsystem of $H$.
As a consequence of Theorem~\ref{SmallestSubsystemBurns}, we have the following correspondence between lazy burning number and dimension.
\begin{theorem}
\label{thm:dim_lazy_equiv}
    For any Steiner triple system $H$, 
   $\dim(H)=b_L(H)-1$.
\end{theorem}

While any set of burning vertices in a Steiner triple system, $H$, propagates to a subsystem, it is not necessarily the case that the lazy burning number of a subsystem is less than the lazy burning number of $H$.  Clearly any maximal subsystem of $H$ has lazy burning number at least $b_L(H)-1$; otherwise, adding any point not in the subsystem would yield a lazy burning set for $H$ of size less than $b_L(H)$.  However, it is possible that there may exist a maximal proper subsystem with larger lazy burning number.  
This fact has been explored in the context of dimension~\cite{Hilton1977}, leading to the classification of triple systems as {\em degenerate} or {\em non-degenerate}.  While these concepts appear in the literature with respect to the dimension of a Steiner triple system, we state them here in terms of lazy burning number.

\begin{definition}
Let $H$ be a Steiner triple system. If every maximal proper subsystem of $H$ has lazy burning number $b_L(H)-1$, then $H$ is {\em non-degenerate}.  Otherwise, i.e.\ if there is a maximal proper subsystem of $H$ with lazy burning number at least $b_L(H)$, $H$ is {\em degenerate}.
\end{definition}

In fact, the only known non-degenerate Steiner triple systems with lazy burning number greater than $3$ are the projective and affine triple systems~\cite{HiltonTeirlinck1980}.  In the case of a degenerate system $H$, Hilton~\cite{Hilton1977} showed that the difference between the lazy burning number of a subsystem and the lazy burning number of $H$ can be arbitrarily large.

\begin{theorem}[\cite{Hilton1977}]
For any integers $d \geq 3$ and $e \geq 3$, there is a Steiner triple system $H$ with a subsystem $H'$ such that $b_L(H)=d$ and $b_L(H')=e$.
\end{theorem}

\noindent
In the particular case that $d=3$, Zeitler~\cite{Zeitler1985} proved that every STS$(v)$ can be embedded in an STS$(2v+1)$ with lazy burning number $3$.

Despite the potential existence of a subsystem with large lazy burning number, the next theorem shows that whenever a Steiner triple system $H$ has lazy burning number $b_L(H) \geq 3$, then $H$ must have {\em some} subsystem with lazy burning number that is one less than $b_L(H)$.

\begin{theorem}
\label{theorem-subSTS}
If $H$ is a Steiner triple system and $b_L(H) \geq 3$, then $H$ contains a proper subsystem $H'$ such that $b_L(H') = b_L(H) - 1$.
\end{theorem}

\begin{proof}
Let $S$ be a minimum lazy burning set for $H$,
fix a point $x \in S$, and let $S_x = S \setminus \{x\}$ and $H'=\langle S_x \rangle$.
If $x\in H'$, then $H' = H$, which would contradict the minimality of $S$, so $x \notin H'$. 
Clearly, $b_L(H') \leq |S_x|$.  If it were the case that $b_L(H') < |S_x|$, then adding $x$ to a minimum lazy burning set for $H'$ would generate a lazy burning set for $H$ of size less than $b_L(H)$, which is a contradiction.  Hence $b_L(H') = |S_x| = |S| - 1 = b_L(H) -1.$
\end{proof}

\begin{corollary}
\label{corollary-subSTS}
If $H$ is a Steiner triple system and $b_L(H) \geq 3$, then $H$ contains at least $b_L(H)$ distinct proper subsystems $H'$ such that $b_L(H') = b_L(H) - 1$.
\end{corollary}

\begin{proof}
    Let $S$ be a minimum lazy burning set for $H$. For each $x\in S$, let $S_x=S\setminus\{x\}$, and consider the subsystem $\langle S_x\rangle$, which has lazy burning number $b_L(H)-1$ by the proof of Theorem~\ref{theorem-subSTS}. Indeed, for any two distinct vertices $x,y\in S$, $\langle S_x\rangle\neq \langle S_y\rangle$, since $y\in \langle S_x\rangle$, but $y\notin \langle S_y\rangle$. 
\end{proof}

\subsection{Spectrum of Existence}
The literature on subsystems and dimension of Steiner triple systems, combined with Theorem~\ref{thm:dim_lazy_equiv}, provides a wealth of results on the existence of Steiner triple systems with given lazy burning number.  Doyen~\cite{STS_with_no_subsystems} proved that for every order $v \equiv 1$ or $3\pmod{6}$, there is an STS$(v)$ with no nontrivial subsystems.  We thus have the following.
\begin{theorem}[\cite{STS_with_no_subsystems}] \label{LazyBurning3}
For every positive integer $v \geq 7$ with $v \equiv 1$ or $3\pmod{6}$, there exists an STS$(v)$ with lazy burning number $3$.
\end{theorem}
Most Steiner triple systems of small order are known to have lazy burning number $3$.
\begin{theorem}[\cite{Doyen1970, TeirlinckThesis}]
If $H$ is a Steiner triple system of order $v \leq 27$, then $b_L(H)>3$ if and only if $H$ is $\PG(3,2)$ (the projective STS$(15)$) or $\AG(3,3)$ (the affine STS$(27)$).  Moreover, if $v \in \{33,37\}$, then every STS$(v)$ has lazy burning number $3$.
\end{theorem}
Nevertheless, Steiner triple systems with higher lazy burning number exist widely.

\begin{theorem}[\cite{TeirlinckThesis}]\label{LazyBurning4}
There exists an STS$(v)$ with lazy burning number at least $4$ for every $v \in \{15,27,31,39\}$ and for all admissible $v \geq 45$, except possibly if $v \in \{51,67,69,145\}$.
\end{theorem}

The following result shows that, in fact, there exists an STS$(v)$ with lazy burning number exactly $4$ for each order $v$ given in Theorem~\ref{LazyBurning4}.

\begin{theorem}[\cite{Teirlinck1979}]
If there is an STS$(v)$ with lazy burning number $\ell$, then there is an STS$(v)$ with lazy burning number $\ell'$ for every $3 \leq \ell' \leq \ell$.
\end{theorem}

Hilton and Teirlinck~\cite{HiltonTeirlinck1980} noted that the dimension of projective and affine triple systems $\PG(n,2)$ and $\AG(n,3)$ equals $n$.
Zeitler~\cite{Zeitler1987} gave a formal proof in the projective case.

\begin{theorem}[\cite{Zeitler1987}]
\label{thm:proj_lazy_solved}
For $n\geq 1$,
$b_L(\PG(n,2))=n+1$.
\end{theorem}

In Section~\ref{sec:Affine and Projective Geometries} we will give a proof of the lazy burning number of affine as well as projective triple systems in the language of lazy burning.

An upper bound on the lazy burning number follows as an immediate consequence of Proposition~3.1 in~\cite{NagySzemeredi}, which was concerned with bounding the size of minimal spreading sets in Steiner triple systems.  An independent proof of this theorem using the language of lazy burning can be found in~\cite{mythesis}.
\begin{theorem}[\cite{mythesis,NagySzemeredi}]
\label{lazy_upper_bound_sts}
If $H$ is an STS$(v)$, then $b_L(H)\leq \lfloor \log_2(v+1)\rfloor$.
\end{theorem}

Theorem~\ref{lazy_upper_bound_sts} shows that a Steiner triple system with lazy burning number $\ell$ must have order at least $\Omega(2^{\ell})$. Moreover, since $\PG(\ell-1,2)$ has $2^{\ell}-1$ points, Theorem~\ref{lazy_upper_bound_sts}, in conjunction with Theorem~\ref{thm:proj_lazy_solved}, implies $\PG(\ell-1,2)$ is the smallest Steiner triple system which has lazy burning number $\ell$. Theorem~14 of~\cite{Hilton1977} shows that a Steiner triple system with given lazy burning number exists whenever the order is sufficiently large.

\begin{theorem}[\cite{Hilton1977}]
\label{thm:AJWH1977}
Let $\ell \geq 4$ and let 
\[
v_\ell = \left\{ \begin{array}{ll} 
2^{2\ell+2} - 3 \cdot 2^{\ell+1} -1, & \mbox{if $\ell$ is even,} \\
2^{2\ell+2}-5\cdot 2^{\ell+1}+3, & \mbox{if $\ell$ is odd.}
\end{array}\right.
\]
If $v>v_\ell$ is admissible, then there exists an STS$(v)$ with lazy burning number $\ell$.
\end{theorem}

\subsection{The Virtue of Laziness}
A natural question is to compare the lazy burning number to the burning number of a Steiner triple system.  When restricted specifically to Steiner triple systems, Corollary 2.7 of \cite{introducing_hyp_burn} implies that laziness is beneficial in terms of the number of sources required to burn the system, which we formalize in the following theorem. 

\begin{theorem} [\cite{introducing_hyp_burn}]
\label{cor:lazyburningUpperbound}
If $H$ is a Steiner triple system, then $b_L(H) < b(H)$.
\end{theorem}

Whereas Theorem~\ref{thm:proj_lazy_solved} gives the lazy burning number of $\PG(n,2)$, in Theorem~\ref{cor:proj_regular_burning_solved} we will determine the burning number of $\PG(n,2)$; these results together show that the projective triple systems provide an infinite family of triple systems whose burning number is exactly one more than the lazy burning number.
We note, however, that the difference between $b(H)$ and $b_L(H)$ can be arbitrarily large for a Steiner triple system $H$.  Theorem~\ref{small_lem} implies that as the order of a Steiner triple system grows, so too does its burning number. On the other hand, Theorem~\ref{LazyBurning3} states that Steiner triple systems with lazy burning number $3$ exist for all admissible orders.  We therefore have the following.

\begin{theorem}
    For each $k \in \mathbb{N}$, there is a Steiner triple system $H$ such that $b(H)-b_L(H) \geq k$.  
\end{theorem}

\section{Projective and Affine triple systems}
\label{sec:AffineProjectiveA}
\label{sec:subsystems}

In this section, we consider burning and lazy burning of affine and projective planes. 
As discussed in Section~\ref{sec:LazyBurning}, lazy burning of these designs has been previously considered in the context of dimension.  
Theorem~\ref{thm:affi_lazy_solved} gives the lazy burning number for affine planes; this result was previously mentioned by 
Hilton and Teirlinck~\cite{HiltonTeirlinck1980}, without proof. 
Theorem~\ref{thm:proj_lazy_solved}  gives the lazy burning number of projective planes, which was proved in the context of dimension by Zeitler~\cite{Zeitler1987}; we give an alternate proof in terms of lazy burning. As far as we know, the corresponding results for round-based burning are new.

In the following, given a prime power $q$, we let $\mathbb{F}_q$ denote the finite field of order $q$. 
Addition of vectors in $\mathbb{F}_q^n$ is defined in the standard way: if ${x}=(x_1,\ldots,x_n)$ and ${y}=(y_1,\ldots,y_n)$, then ${x}\oplus {y}=(x_1+y_1,\ldots,x_n+y_n)$ where each $x_i+y_i$ is addition in $\mathbb{F}_q$. 
The zero element will be denoted $\mathbf{0}_n$, the all-ones vector $(1,1\ldots,1)$ is denoted $\mathbf{1}_{n}$, and ${e}_i$ is the standard basis vector with $1$ in the $i^{\rm th}$ coordinate and $0$ elsewhere. 

\begin{definition}
    The \emph{projective Steiner triple system} $\PG(n,2)$ has vertex set $\mathbb{F}_2^{n+1}\setminus\{\mathbf{0}_{n+1}\}$ and contains the triple $\{x,y,x\oplus y\}$ for each distinct pair $x,y\in \mathbb{F}_2^{n+1}\setminus\{\mathbf{0}_{n+1}\}$. 
\end{definition}
The projective Steiner triple system $\PG(n,2)$ is a Steiner triple system on $2^{n+1}-1$ vertices. 
Note that $\PG(n,2)$ contains $\PG(n-1,2)$ as a (maximal) subsystem with point set $(\mathbb{F}_2^{n}\setminus\{\mathbf{0}_n\}) \times\{0\}$.

\begin{definition}
    The \emph{affine Steiner triple system} $\AG(m,3)$ has vertex set $\mathbb{F}_3^{m}$ and contains the triple $\{x,y,z\}$ whenever $x\oplus y\oplus z = \mathbf{0}_m$ for distinct $x,y,z\in \mathbb{F}_3^{m}$. 
\end{definition}

The affine Steiner triple system $\AG(m,3)$ is a Steiner triple system on $3^{m}$ vertices. Note that all maximal proper subsystems of $\AG(m,3)$ are isomorphic to $\AG(m-1,3)$.

\label{sec:Affine and Projective Geometries}

We now provide an alternate proof of Theorem~\ref{thm:proj_lazy_solved}
(originally established by Zeitler~\cite{Zeitler1987}) using the language of lazy burning.

{
\renewcommand{\thetheorem}{\ref{thm:proj_lazy_solved}}
\begin{theorem}
For $n\geq 1$,
$b_L(\PG(n,2))=n+1$.
\end{theorem}\addtocounter{theorem}{-1}
}

\begin{proof}
We show the upper bound $b_L(\PG(n,2))\le n+1$ by induction.
For the base case,  $\PG(1,2)$ is a single triple and so has lazy burning number 2.
We now suppose that $b_L(\PG(n-1,2))\le n$ and let $S$ be a lazy burning set of size at most $n$ for $\PG(n-1,2)$ on vertex set $(\mathbb{F}^{n}_2\setminus\{\mathbf{0}_n\})\times\{0\}$. Thus $S\cup \mathbf{1}_{n+1}$ is a lazy burning set of size at most $n+1$ for $\PG(n,2)$. 

We now establish the lower bound $b_L(\PG(n,2))\ge n+1$. 
For any subset $S$ of the vertices in $\PG(n,2)$, the set $\langle S\rangle$ forms a linear subspace of $\mathbb{F}_2^{n+1}$  (recall that $\langle S \rangle$ is the set of points that become burned if the points of $S$ are burned), which contains at most $2^{|S|}$ elements. 
If $|S|<n+1$, then the number of vertices burning after applying the lazy burning process to $S$ is less than $2^{n+1}$, and so $S$ could not be a lazy burning set. Thus a lazy burning set must have size at least $n+1$. 
\end{proof}

\begin{corollary}
\label{arbitrary_lazy_num}
For all $n\in\mathbb{Z}^+\setminus\{1\}$, there exists a Steiner triple system with lazy burning number $n$.
\end{corollary}

We now prove the corresponding result for affine triple systems, which was previously mentioned by Hilton and Teirlinck~\cite{HiltonTeirlinck1980} without proof.   
\begin{theorem}
\label{thm:affi_lazy_solved}
For $n\geq 1$, $b_L(\AG(m,3))=m+1$.
\end{theorem}
\begin{proof}
The points of $\AG(m,3)$ form the field $\mathbb{F}_3^m$. 
We note that for all $u$, if $S$ is a lazy burning set for $\AG(m,3)$, then so is $S_u=\{ s \oplus u : s\in S \}$. 
To see this, observe that $(s_1\oplus u) \oplus (s_2 \oplus u) = (s_1\oplus s_2)\oplus (u\oplus u) = (s_1 \oplus s_2) \oplus 2u$. 
It follows that if point $v$ is ever burned when starting with lazy burning set $S$, then $v\oplus 2u$ will be burned when starting with lazy burning set $S_u$. 
Since all vertices are burned when starting with lazy burning set $S$, it follows that all vertices are burned when starting with lazy burning set $S_u$. 
We may thus assume that there exists a minimum lazy burning set $S$ that includes the point $\mathbf{0}_m$. 

Given a subset $U$ of points in $\AG(m,3)$ which contains $\mathbf{0}_m$, we show that $\langle U \rangle$ is a linear subspace of $\mathbb{F}_3^m$.  If $u$ and $v$ are distinct elements of $\langle U\rangle$, then $\{u,v,u\oplus v\}$ is a block of $\AG(m,3)$, so that $u \oplus v$ burns, i.e.\ $u \oplus v \in \langle U \rangle$.  In particular, since $\mathbf{0}_m, u \in \langle U \rangle$, it follows that $2u\in\langle U\rangle$.  Clearly $0u$ and $1u$ are also in $\langle S\rangle$, and so closure by scalar multiplication is satisfied. 

Since $\mathbb{F}_3^m$ has dimension $m$ as a vector space, it has a basis $L$ of cardinality $m$.  Let $S=L \cup \{\mathbf{0}_m\}$.  We show that $S$ is a lazy burning set of cardinality $m+1$.  For any $u \in L$, as previously noted we have that $2u \in \langle S \rangle$.  Hence, since $\langle S \rangle$ is closed under addition, the subspace $\langle S \rangle$ contains the span of $L$, which is $\mathbb{F}_3^m$. 

If there exists a lazy burning set $B$ of cardinality $m$, by our previous argument we may assume that $\mathbf{0}_m \in B$.  It follows that $B \setminus \{\mathbf{0}_m\}$ is a spanning set for $\mathbb{F}_3^m$ of cardinality $m-1$, which is a contradiction.  Hence $S$ is a minimum lazy burning set.
\end{proof}

Before proceeding to consider the burning numbers of $\PG(n,2)$ and $\AG(n,3)$, we first define a modified type of burning sequence, in which the arsonist is not required to burn a new source in every round (in other words, the arsonist may pass); this idea will be useful in the subsequent theorems.  Specifically, let $U=(u_1, u_2, \ldots, u_k)$ be a sequence, where for each $\rho \in \{1, \ldots, k\}$, either $u_{\rho}$ is a vertex which is valid at round $\rho$ or $u_{\rho} = *$, where $*$ denotes the decision to pass.  We call the sequence $U$ a {\em generalized burning sequence}
if it leaves the hypergraph completely burned when the arsonist chooses $u_{\rho}$ as a source in all rounds $\rho$ for which $u_{\rho}$ is a valid vertex, and the arsonist burns no vertex in any round $\rho$ for which $u_{\rho}=*$.  

\begin{example} \label{GenBurnSeq-STS(9)}
The sequence $(1,2,4,*,*)$ is a generalized burning sequence for the STS$(9)$ with the following blocks.
\[
\begin{array}{llll}
\{1,2,3\} & \{1,4,7\} & \{1,5,9\} & \{1,6,8\} \\
\{4,5,6\} & \{2,5,8\} & \{2,6,7\} & \{2,4,9\} \\
\{7,8,9\} & \{3,6,9\} & \{3,4,8\} & \{3,5,7\} 
\end{array}
\]
In the first two rounds, only points $1$ and $2$ burn.  In the third round, point $3$ becomes burned by propagation and the arsonist burns point $4$.  In the fourth round, $7$, $8$ and $9$ burn by propagation, but the arsonist does not burn an additional source.  Finally, in the fifth round, $5$ and $6$ burn by propagation.  
\end{example}

The following lemma shows that introducing passing moves does not change the burning number.

\begin{lemma} \label{generalized burning sequence}
If $H$ has a generalized burning sequence of length $k$, then $b(H) \leq k$.
\end{lemma}

\begin{proof}
Let $U=(u_1, \ldots, u_k)$ be a generalized burning sequence for $H$.  Note that if for some $\rho$, $u_{\rho}=*$ and no vertices burn by propagation in round $\rho$, then removing this entry from $U$ creates a generalized burning sequence of length less than $k$.  Thus, we may assume that there is no such entry. 

Form a sequence $U'=(u_1',u_2',\ldots,u_k')$ as follows.  For $\rho \in \{1,\ldots,k\}$, if $u_{\rho}=*$, let $u_{\rho}'$ be any vertex that burns by propagation in round $\rho$; otherwise, set $u_{\rho}'=u_{\rho}$.  Then $U'$ is a burning sequence for $H$ of length $k$ (in which the $\rho^{\mathrm{th}}$ entry is a redundant source if $u_{\rho}=*$), so $b(H) \leq k$.
\end{proof}

We now determine the burning numbers of $\PG(n,2)$ and $\AG(n,3)$; generalized burning sequences will be used in the proofs of the upper bound on $b(\PG(n,2))$ and $b(\AG(n,3))$.  

\begin{theorem}
\label{cor:proj_regular_burning_solved}
If $n \geq 2$, then $b(\PG(n,2))=n+2$.
\end{theorem}
\begin{proof}
    To show that $b(\PG(n,2)) \leq n+2$, we proceed by induction on $n$.
    For the base case, we note that $\PG(2,2)$ is isomorphic to the unique STS$(7)$, which is, up to isomorphism, given below.
    \[
    \{0,1,3\}\quad \{1,2,4\}\quad \{2,3,5\}\quad \{3,4,6\}\quad \{0,4,5\}\quad \{1,5,6\} \quad\{0,2,6\}.
    \]
    It is not hard to see that $(0, 1, 2, 4)$ is a minimum burning sequence, so $b(\PG(2,2))=4$.
    We further note that the last source, $4$, is redundant, so that $(0,1,2,*)$ is a minimum generalized burning sequence for $\PG(2,2)$.
		
    For our inductive hypothesis we will now assume that $b(\PG(n-1,2)) \leq n+1$ and that $\PG(n-1,2)$ has a generalized burning sequence of length $k \leq n+1$ in which the $k^{\mathrm{th}}$ entry is $*$. 
    Let $\mathcal{D}$ be the subsystem of $\PG(n,2))$ with point set $(\mathbb{F}_2^{n} \setminus \{\mathbf{0}_n\})\times \{0\}$, and note that $\mathcal{D}$ is isomorphic to $\PG(n-1,2)$. Let $S=(x_1,\ldots,x_{k-1},*)$ be a generalized burning sequence for $\mathcal{D}$, where $k \leq n+1$; such a sequence exists by the inductive hypothesis.
    We let $S' = (x_1,\ldots,x_{k-1}, e_{n+1})$ and note that by using the burning sequence $S'$, all of $\mathcal{D}$ will have burned by round $k$. Further, for every $x\in\mathcal{D}$, we have the block $\{x\times \{0\}, x\times \{1\}, e_{n+1}\}$, and so every vertex will burn in the next round. Thus, $S''=(x_1,\ldots,x_{k-1},e_{n+1},*)$ is a generalized burning sequence for $\PG(n,2))$.  By Lemma~\ref{generalized burning sequence}, $b(\PG(n,2)) \leq k+1 \leq n+2$.  

For the lower bound, note that by Corollary~2.7 of~\cite{introducing_hyp_burn} and Theorem~\ref{thm:proj_lazy_solved}, $b(\PG(n,2)) > b_L(\PG(n,2))=n+1$.
\end{proof}

\begin{theorem}
\label{th:affine_burning_number}
    For all $m \geq 2$, $b(\AG(m,3)) = m+3$.
\end{theorem}
\begin{proof}
We first prove that $b(\AG(m,3)) \leq m+3$ by induction on $m$.  As a base case, consider $m=2$, so we need to show that $b(\AG(2,3))\leq 5$.  Note that $\AG(2,3)$ is isomorphic to the unique STS$(9)$.  Using the representation of this system given in Example~\ref{GenBurnSeq-STS(9)}, we have that $(1,2,4,*,*)$ is a generalized burning sequence of length $5$.  

Now suppose that $b(\AG(m-1,3)) \leq m+2$, and that there is a generalized burning sequence $S$ of length $s \leq m+2$ of the form $S=(v_1,v_2, \ldots, v_{s-2},*,*)$.  Consider $\AG(m,3)$ on point set $\mathbb{F}_3 \times \mathbb{F}_3^{m-1} \cong \mathbb{F}_3^m$.  Let $H \cong \AG(m-1,3)$ be the subsystem of $\AG(m,3)$ with point set $\{0\} \times \mathbb{F}_3^{m-1}$, and let $z \in \{1\} \times \mathbb{F}_3^{m-1}$. 
We claim that $(v_1, \ldots, v_{s-2},z,*,*)$ is a generalized burning sequence for $\AG(m,3)$.   Note that in round $s$, the vertex $z'=-(v_1\oplus z) \in \{2\} \times \mathbb{F}_{3}^{m-1}$ burns. Now, by the end of round $s$, by the induction hypothesis, every vertex of $H$ will already have burned.  Furthermore, each vertex of $\{1\} \times \mathbb{F}_3^{m-1}$ has the form $-(h\oplus z')$ for some $h \in \{0\}\times \mathbb{F}_3^{m-1}$ and each vertex of $\{2\} \times \mathbb{F}_3^{m-1}$ has the form $-(h\oplus z)$ for some $h \in \{0\}\times \mathbb{F}_3^{m-1}$.  Thus, every remaining vertex will burn in round $s+1$.  
Thus by induction we see that $\AG(m-1,3)$ has a generalized burning sequence of length at most $m+2$, and so by Lemma~\ref{generalized burning sequence}, $b(\AG(m-1,3)) \leq m+2$. 

For the lower bound, consider the subset $W$ of points that are burning after $m$ rounds.  Since $b_L(\AG(m,3))=m+1$ by Theorem~\ref{thm:affi_lazy_solved} and only $m$ sources have been chosen, the subsystem $\langle W \rangle$ generated by $W$ is a proper subsystem of $\AG(m,3)$.  Thus $W$ must be contained in some maximal proper subsystem $A\cong \AG(m-1,3)$ of $\AG(m,3)$. Without loss of generality, we assume that the point set of $A$ is $\mathbb{F}_3^{m-1} \times \{0\}$.  

Consider the first time a source $x$ outside $A$ is chosen, in round $\rho \geq m+1$.  Note that $x \notin \mathbb{F}_3^{m-1} \times \{0\}$, so without loss of generality we may suppose that $x \in \mathbb{F}_3^{m-1} \times \{1\}$.  In round $\rho+1$, the only vertices outside $A$ that are burned by propagation are those points $z$ in blocks of the form $\{x,y,z\}$ where $y \in \mathbb{F}_3^{m-1} \times \{0\}$; thus $z \in \mathbb{F}_3^{m-1} \times \{2\}$, so that no point of $(\mathbb{F}_3^{m-1} \times \{1\}) \setminus \{x\}$ can burn until at least round $\rho+2$.  It follows that the earliest round in which all points of $\AG(m,3)$ can burn is $m+3$.
\end{proof}

Corollary~2.7 of~\cite{introducing_hyp_burn} implies that any Steiner triple system $H$ of order $v \geq 3$ satisfies $b(H) \geq b_L(H)+1$.  This lower bound on the burning number is met for the projective triple systems but not for the affine triple systems, which have $b(\AG(m,3)) = b_L(\AG(m,3))+2$.

\section{Discussion}
\label{sec:Discussion}

In this paper, we have considered the burning and lazy burning processes on hypergraphs as they apply to Steiner triple systems.  There remain a number of avenues for further exploration on this subject, which we now discuss.

Theorem~\ref{thm:ConfinedBurningNumber} gives lower and upper bounds for the burning number of an STS$(v)$. The upper bound is tight, while the lower bound is tight to within a small additive constant; one could seek to eliminate this constant.  Further, these bounds provide an interval of possible burning numbers of an STS$(v)$, so one might ask if is it possible to find an STS$(v)$ with burning number $b$ for each value $b$ in this interval. 

In Section~\ref{sec:LazyBurning}, we summarized the literature on dimension of Steiner triple systems in the context of lazy burning. One salient open question is how the subsystem structure impacts the lazy burning number. Degenerate systems, which have subsystems with lazy burning number at least as big as the larger system, show that this interaction is non-trivial.

In Section~\ref{sec:AffineProjectiveA}, we considered the burning number of affine and projective triple systems. 
It is well known that the projective triple systems can be generated by repeated application of the so-called {\em doubling construction}. The doubling construction (technically the double-and-add-one construction) takes an STS$(v)$ and produces an STS$(2v+1)$; see \cite[Construction 2.15]{designs_handbook}. 
Similarly, the affine triple systems can be generated by repeated application of the so-called {\em tripling construction}; see \cite[Construction 2.17]{designs_handbook}, which produces an STS$(3v)$ from a given STS$(v)$. 
Zeitler~\cite{Zeitler1987} showed that the projective triple systems are the only non-degnerate STSs whose lazy burning numbers increase on doubling. Thus doubling an affine triple systems does not result in an increase of the lazy burning number. 
Since tripling an affine triple system $\AG(m,3)$ produces an $\AG(m+1,3)$, Theorem~\ref{thm:affi_lazy_solved} shows that tripling an affine triple system increases the lazy burning number by one. We conjecture that, in a similar fashion to projective triple systems with doubling, affine triple systems are the only non-degenerate systems whose lazy burning number increases on tripling.

Theorems~\ref{cor:proj_regular_burning_solved} and~\ref{th:affine_burning_number} also show that doubling a projective triple system or tripling an affine triple system increases the burning number by one.  We ask if these are the only non-degenerate systems whose burning numbers increase on doubling or tripling.
It would also be interesting to consider the effects of doubling and tripling on the (lazy) burning number of more general triple systems.

Other variations of burning have also been studied in the literature. For example, one variant called {\em cooling} asks for a sequence which burns a (hyper)graph as slowly as possible, see \cite{cooling} for details. 
It would be interesting to consider cooling for triple systems, or for hypergraphs more generally. 

One could further consider a spectrum of burning processes, from slowest to fastest. This idea is formalized in {\em liminal} burning~\cite{liminal}, which generalizes both burning and cooling. 
Liminal burning for triple systems, or for hypergraphs more generally, would be another interesting area of enquiry. 

\section*{Acknowledgements}

Authors Burgess, Danziger and Pike acknowledge NSERC Discovery Grant support (grant numbers RGPIN-2025-04633, RGPIN-2022-03816 and RGPIN-2022-03829, respectively) and Jones acknowledges NSERC and AARMS scholarship support.

\end{document}